\documentclass[12pt,a4paper,reqno]{amsart}

\usepackage[a4paper,bindingoffset=0in,%
            left=1.1in,right=1.1in,top=1.5in,bottom=1.5in,%
            footskip=.25in]{geometry}
            
\usepackage{amscd,amssymb}
\usepackage[leqno]{amsmath}
\usepackage{tikz}
\usepackage{tikz-3dplot}
\usepackage{tikz-cd}
\usepackage{faktor}

\newtheorem  {theorem}                  {Theorem}
\newtheorem* {theorem*}                   {Theorem}

\newtheorem {lemma}[theorem] {Lemma}
\newtheorem {prop}[theorem]      {Proposition}
\newtheorem* {prop*}     {Proposition}

\newtheorem {corollary}[theorem]      {Corollary}

\theoremstyle{definition}
\newtheorem {defi}[theorem] {Definition}
\newtheorem {Remark} [theorem]         {Remark}

\newtheorem {Example} [theorem]    {Example}
\newtheorem* {Example*}    {Example}

\def\R{\mathbb{R}}

\newcommand{\pp}[2]{\frac{\partial#1}{\partial#2}}
\newcommand{\tpp}[2]{\tfrac{\partial#1}{\partial#2}}

\newcommand{\len}{\operatorname{length}}

\title{The periodic orbit conjecture for steady Euler flows}

\author{Robert Cardona}\address{ Robert Cardona,
Laboratory of Geometry and Dynamical Systems, Department of Mathematics, Universitat Polit\`{e}cnica de Catalunya and BGSMath Barcelona Graduate School of
Mathematics,  Avinguda del Doctor Mara\~{n}on 44-50, 08028 , Barcelona, Spain  \it{e-mail: robert.cardona@upc.edu}
 }
 \thanks{The author acknowledges financial support from the Spanish Ministry of Economy and Competitiveness, through the Mar\'ia de Maeztu Programme for Units of Excellence in R\& D (MDM-2014-0445) via an FPI grant. The author is partially supported by the grants MTM2015-69135-P/FEDER and PID2019-
103849GB-I00 / AEI / 10.13039/501100011033, and AGAUR grant 2017SGR932.}
\begin{document}

\begin{abstract}
The periodic orbit conjecture states that, on closed manifolds, the set of lengths of the orbits of a non-vanishing vector field all whose orbits are closed admits an upper bound. This conjecture is known to be false in general due to a counterexample by Sullivan. However, it is satisfied under the geometric condition of being geodesible. In this work, we use the recent characterization of Eulerisable flows (or more generally flows admitting a strongly adapted one-form) to prove that the conjecture remains true for this larger class of vector fields. 
\end{abstract}

\maketitle

\section{Introduction}

A classical question in the study of compact foliations on compact manifolds is the existence of an upper bound on the volume of their leaves. For one-dimensional foliations, this was known as the periodic orbit conjecture.\newline

\textit{Periodic orbit conjecture:} Let $X$ be a non-vanishing vector field on a manifold $M$ such that every orbit of $X$ is closed. Then there is an upper bound on the lengths of the orbits of $X$. \newline

This conjecture was proved in dimension three by Epstein \cite{Eps}. In higher dimensions, however, Sullivan constructed a beautiful counterexample on a five-dimensional compact manifold \cite{S1}. A counterexample in the sharpest case of dimension four was settled by Epstein and Vogt \cite{EV} a couple of years later. A theorem by Wadsley \cite{W} shows that a necessary and sufficient geometric condition for the conjecture to hold is that the vector field $X$ is geodesible, i.e. there is some metric making its orbits geodesics. In this work, we show that even if we allow $X$ to be Eulerisable (or more generally to admit a strongly adapted one-form), which is a larger class of vector fields, the conjecture is still satisfied.  A vector field is Eulerisable if it is a steady solution to the Euler equations for some metric. Respectively, a vector field $X$ admits a strongly adapted one-form if there is some one-form $\alpha$ such that $\alpha(X)>0$ and $\iota_Xd\alpha$ is exact.

\begin{theorem}\label{main}
Eulerisable fields (more generally, flows admitting a strongly adapted one-form) on closed manifolds satisfy the periodic orbit conjecture.
\end{theorem}
A corollary of Theorem \ref{main} is a new dynamical obstruction for a vector field to be topologically conjugate to an Eulerisable flow. It was previously known that the presence of plugs is an obstruction to such property \cite{PRT}.
\begin{corollary}\label{coro2}
A volume-preserving non-vanishing vector field $X$ on a compact manifold all whose orbits are closed but whose lengths are not locally bounded, is not topologically conjugate to an Eulerisable flow.
\end{corollary}

Such flows exist in dimension $n\geq 4$, as our analysis of Thurston-Sullivan's example in Section \ref{ss:exTS} shows.\\

In the last section, we take a look at Wadsley's theorem and show that any vector field satisfying the periodic orbit conjecture is in fact a Beltrami-type stationary solution to the Euler equations for some metric. Combining this with the main theorem, we get the following consequence.

\begin{corollary}\label{coro}
Let $X$ be a non-vanishing steady solution to the Euler equations all whose orbits are closed. Then there is some other metric $g$ for which $X$ is a Beltrami type steady solution to the Euler equations.
\end{corollary}

\textbf{Acknowledgements:} The author is grateful to Daniel Peralta-Salas, who proposed this question during the author's stay in Madrid for the Workshop on Geometric Methods in Symplectic Topology in December 2019. Thanks to Francisco Torres de Lizaur for useful comments. The author is grateful to the anonymous referee of this work for several comments which improved this paper.

\section{Currents and Eulerisable fields}

Let $M$ be a closed smooth manifold and consider $\Omega^k(M)$ the space of differential $k$-forms on $M$, which is endowed with the natural $C^\infty$-topology. A $k$-current is a continuous linear function over $\Omega^k(M)$. We denote the space of $k$-currents by $\mathcal{Z}^k(M)$. A concise introduction to the theory of currents is contained in \cite{Gl}, we will give here only some basic background.

Currents are equipped with a ``boundary" operator $\partial:\mathcal{Z}^k(M) \rightarrow \mathcal{Z}^{k-1}(M)$, which is continuous, defined by 
$$ \partial c(\omega)=c(d\omega), $$
where $c$ is a $k$-current, $\omega$ is a $k$-form and $d$ denotes the usual exterior derivative. A current $c$ which has no boundary is called a ``cycle". A classical theorem by Schwartz establishes that the dual space to $\mathcal{Z}^k(M)$ is $\Omega^k(M)$. 

\begin{Example}
An example of a $k$-current is given by integration along a $k$-chain, i.e. $c(\omega)=\int_c \omega$ where $c$ is a $k$-chain. This example will be of our interest, as it is used for the characterization of geodesible and Eulerisable vector fields.
\end{Example}

Let us now fix some non-vanishing vector field $X$ in $M$. To each point $p\in M$, there is an associated Dirac $1$-current
\begin{align*}
\delta_p: \Omega^1(M)&\longrightarrow \mathbb{R}\\
			\alpha &\longmapsto \alpha_p(X_p).
\end{align*}
The closed convex cone in $\mathcal{Z}^1(M)$ generated by Dirac $1$-currents is called the space of foliation currents. A property of such currents is that for a one-form $\alpha$ such that $\alpha(X)>0$, any foliation current $z$ satisfies $z(\alpha)>0$. Another class of $2$-currents that we need to introduce is that of tangent $2$-chains: two-dimensional chains satisfying that the vector field $X$ is tangent to them. We will denote by $\mathcal{B}_X$ the set of boundaries of tangent $2$-chains
$$\mathcal{B}_X=\{ \partial c\enspace | \enspace c \text{ is a tangent } 2\text{-chain of }X\}. $$
The theory of currents and cycles is a useful tool to study foliations in a broad sense, and for one-dimensional foliations, it was used to topologically characterize geodesible vector fields.

\begin{defi}
A vector field $X$ is \textbf{geodesible} if there exists a Riemannian metric such that the orbits of $X$ are geodesics.
\end{defi}
A geometric characterization of geodesible fields, due to Gluck, states that a vector field is geodesible if and only if there is a one-form $\alpha$ such that $\alpha(X)>0$ and $\iota_Xd\alpha=0$. We denote the space of foliation cycles of $X$ by $\mathcal{C}_X$. The topological characterization in terms of currents was proved by Sullivan \cite{S2}.

\begin{theorem}\label{thm:sul}
A non-vanishing vector field on a closed manifold is geodesible if and only if there is no sequence of tangent $2$-chains whose boundary approximates a foliation cycle (i.e. $\overline{\mathcal{B}_X}\cap \mathcal{C}_X=\{0\}$).
\end{theorem}

Let us elaborate on the concrete meaning of ``approximation" in this context. This condition says that there is a sequence of tangent $2$-chains $c_k$ such that $\partial c_k$ converges to some foliation cycle $z$ in the weak topology. Convergence in the weak topology can be described \cite[Section 4.2]{Morg} by the condition that for any one-form $\gamma$ we have $\partial c_k(\gamma)\rightarrow z(\gamma)$.

In \cite{PRT}, a wider class of non-vanishing vector fields related to the stationary Euler equations in hydrodynamics is introduced and characterized using this very same language. The stationary Euler equations on a Riemannian manifold $(M,g)$, which model an ideal fluid in equilibrium, can be expressed in the dual language of forms:
\begin{align*}
\begin{cases}
\iota_Xd\alpha=-dB\\
d\iota_X\mu=0
\end{cases},
\end{align*}
where $X$ is the velocity field of the fluid, the one-form $\alpha=g(X,\cdot)$ is dual to $X$ with respect to the metric, and $\mu$ is the Riemannian volume form. The function $B$ is called the Bernoulli function, and it is constant if and only if $\iota_Xd\alpha$ vanishes. Confer \cite{AK} for an introduction to geometrical and topological aspects of hydrodynamics.

\begin{defi}\label{eulchar}
Let $M$ be manifold with a volume form $\mu$. A volume-preserving vector field $X$ is \textbf{Eulerisable} if there is a metric $g$ on $M$ for which $X$ satisfies the stationary Euler equations for some Bernoulli function $B:M \rightarrow \R$.
\end{defi}
Eulerisable fields which are not geodesible are constructed in \cite{CV}, proving that it is a wider class of vector fields than volume-preserving geodesible fields. Non-vanishing Eulerisable fields have both a geometric and a topological characterization. Consider the following linear subspace of $1$-currents: 
$$ \mathcal{F}_{d\alpha}=\{ \partial c \enspace | \enspace  c \text{ is a 2-chain s.t. } \int_c d\alpha=0\}.$$
Below, we give a partial statement of Theorem 5.2 in \cite{PRT} which is enough for our purposes. The complete statement is a generalization of Sullivan's characterization of geodesible fields.
\begin{theorem}[\cite{PRT}]\label{thm:eulerisable}
Let $X$ be a non-vanishing volume-preserving vector field on a closed manifold of dimension $n\geq 3$. Then $X$ is Eulerisable if and only if there is some $\alpha$ so that $\alpha(X)>0$ and $\iota_Xd\alpha$ is exact. Furthermore, if $X$ is Eulerisable with one-form $\alpha$, no sequence of elements in $\mathcal{F}_{d\alpha}$ approximates a non-trivial foliation cycle of $X$ (i.e. $\overline{\mathcal{F}_{d\alpha}}\cap \mathcal{C}_X=\{ 0 \}$).
\end{theorem}

\begin{Remark}
As introduced in \cite{Tao}, a non-vanishing vector field $X$ admits a strongly adapted one-form $\alpha$ if $\alpha(X)>0$ and $\iota_Xd\alpha$ is exact. An equivalent statement applies to this class of vector fields. If a non-vanishing vector field $X$ (not necessarily volume-preserving) on a closed manifold of dimension $n\geq 3$ admits a strongly adapted one-form $\alpha$, then no sequence of elements in $\mathcal{F}_{d\alpha}$ approximates a non-trivial foliation cycle of $X$.
\end{Remark}

\section{The periodic orbit conjecture for Eulerisable fields}

In this section, we first recall a counterexample to the periodic orbit conjecture given by Thurston and show that it is volume-preserving (and even Beltrami for some metric). Following \cite{C1}, recall that for a given vector field $X$ on a Riemannian manifold $(M,g)$ of odd dimension $2n+1$ we define its curl as the unique vector field $\omega$ satisfying the equation
\begin{equation}\label{eq:curl}
 \iota_\omega \mu=(d\alpha)^n,
 \end{equation}
where $\alpha=g(X,\cdot)$ is the one-form dual to $X$ by the metric. 

\begin{defi}
A vector field $X$ in an odd-dimensional Riemannian manifold $(M,g)$ is said to be a Beltrami field\footnote{Especially in three dimensions, it is sometimes required in the definition that the vector field preserves the Riemannian volume. In our discussion, we check that Sullivan-Thurston's is Beltrami and volume-preserving.} if it is everywhere parallel to its curl, i.e. $\omega=fX$ for some function $f\in C^\infty(M)$.
\end{defi}
In three dimensions, every volume-preserving Beltrami field is a stationary solution to the Euler equations. However, in higher odd dimensions there are volume-preserving Beltrami which are not Eulerisable \cite{C1}. On a manifold of any dimension (even or odd, at least three), we say that a stationary solution to the Euler equations is of Beltrami-type if its corresponding Bernoulli function is constant.

\subsection{Sullivan-Thurston's example}\label{ss:exTS}
We follow \cite{MS} to give an explicit description of Thurston's analytic counterexample to the periodic orbit conjecture. This example was inspired by Sullivan's paper \cite{S1}.\\

Let $H$ be the Heisenberg group with parameters $(x,y,z)\in \mathbb{R}^3$.
\begin{equation*}
H=\left(
\begin{matrix}
1 & x & z \\
0 & 1 & y \\
0 & 0 & 1
\end{matrix}\right)
\end{equation*}
The following lattice induces an action on $H$ via left matrix multiplication, where $(a,b,c) \in \mathbb{Z}^3$.
\begin{equation*}
\Lambda=\left(
\begin{matrix}
1 & a & c \\
0 & 1 & b \\
0 & 0 & 1
\end{matrix}\right)
\end{equation*}
We denote by $\pi$ the projection from $H\times \R^2$, which is equipped with coordinates $(x,y,z,t,u)$, to $M=H / \Lambda \times S^1\times S^1$. The vector fields 
\begin{align*}
V_1&=\cos t\pp{}{x}+ \sin t\left(\pp{}{y} + x\pp{}{z}\right)\\
V_2&= -\sin t \pp{}{x} + \cos t \left(\pp{}{y} +x\pp{}{z}\right)
\end{align*}
form together with $\pp{}{z}, \pp{}{t}, \pp{}{u}$ a global set of independent vector fields. 

Thurston's example is given by the vector field induced in the quotient space $M$ by
\begin{align*}
X&= \sin (2u)V_1 + 2\sin^2u\pp{}{t} - \cos^2u\pp{}{z} \\
&= \sin(2u)\cos t \pp{}{x} + \sin(2u)\sin t\left(\pp{}{y} + x \pp{}{z}\right) + 2\sin^2u\pp{}{t} - \cos^2u \pp{}{z}.
\end{align*} 
Let us check that $X$ is well defined in the quotient space. Denote by $\varphi_{a,b,c}$ the action of an element $(a,b,c)\in \Lambda$, with the obvious identification between points in $\mathbb{Z}^3$ and matrices in $\Lambda$. This map is given in coordinates by
\begin{align*}
\varphi_{a,b,c}:\mathbb{R}^3 &\longrightarrow \mathbb{R}^3\\
		(x,y,z) &\longmapsto (x',y',z')=(x+a,y+b,z+ay+c).
\end{align*}
From this component-wise expression, it is clear that $\pp{}{x}$ and $\pp{}{z}$ descend to the quotient. Computing the pushforward of the remaining term $\pp{}{y}+x\pp{}{z}$ we get
\begin{align*}
(\varphi_{a,b,c})_*\left(\pp{}{y}+x\pp{}{z}\right)&= \pp{}{y'}+a\pp{}{z'}+(x'-a)\pp{}{z'}\\
&=\pp{}{y'}+x'\pp{}{z'}.
\end{align*}
This proves that this term and the vector field $X$ descend to the quotient by the left action of $\Lambda$.\\

On $U=\{u\neq 0 \mod \pi \}$, the orbits of $W= \frac{1}{2\sin^2u} X$ are all closed of period $2\pi$. This can be checked by direct computation of the flow of $W$, which is $2\pi$-periodic \cite{MS}. In particular, as $u$ approaches $0$ or $\pi$, the orbits of $X$ have arbitrarily large periods. Furthermore, $X$ extends as $-\pp{}{z}$ along $u=0 \mod \pi$. Since there is an element in $\Gamma$ acting by translations along $z$, we deduce that the orbits of $X$ at $u= 0\mod \pi$ are also closed. The volume form $\mu=dx\wedge dy\wedge dz\wedge dt \wedge du$, which descends to $M$, is preserved by $X$:
\begin{align*}
d\iota_X \mu&= d\big[\sin(2u)\cos tdy\wedge dz\wedge dt\wedge du - \sin (2u)\sin t dx\wedge dz\wedge dt\wedge du \\ 
&+ (x\sin t \sin(2u) - \cos^2u)dx\wedge dy \wedge dt\wedge du 
-2\sin^2u dx\wedge dy \wedge dz\wedge du\big]\\
&=0.
\end{align*}
This shows that this counterexample is volume-preserving. In fact it is even Beltrami for some Riemannian metric. The one-form $\beta= \frac{1}{2}dt-dz + xdy$ is well defined in the quotient space since
\begin{align*}
\varphi_{a,b,c}^*(-dz'+x'dy')&=-(dz+ady)+(x+a)dy\\
			&=-dz+xdy.
\end{align*}
Furthermore, it satisfies
\begin{equation*}
\begin{cases}
\beta(X)= \sin^2u + \cos^2u=1 \\
(d\beta)^2=0.
\end{cases} 
\end{equation*}

It is now standard to construct a metric (cf. the last part of the proof of Proposition \ref{prop:s1euler}) such that the vector field $X$ has a vanishing curl and preserves the Riemannian volume: it is a volume-preserving Beltrami field. The fact that the curl of $X$ vanishes follows from the fact that $(d\beta)^2=0$ and the defining equation \eqref{eq:curl} of the curl of a vector field. A vector field with vanishing curl is called an irrotational (Beltrami) vector field.
\subsection{Proof of the main theorem}

We proceed with the proof of Theorem \ref{main}. We will make a few technical modifications to the argument in \cite[Section 5]{PRT} (which is inspired by the geodesible case \cite[Section 1.4.1]{R}) to prove that plugs are not Eulerisable. The strategy consist in finding a suitable family of tangent two chains whose boundary approximates a foliation cycle and whose flux (with respect to the adapted form) tends to zero. This will lead to a contradiction. These tangent $2$-chains look like ``rectangular sheets" in \cite{PRT}, in our proof instead, they are cylindrical. \newline

Let $M$ be a compact manifold, and let $X$ be a non-vanishing vector field all whose orbits are closed. These orbits define a one-dimensional foliation by compact leaves. Let us define the set 
 $$B_1= \{p\in M\enspace | \enspace \text{the length of the orbits is not bounded near }p\}.$$
 This set is referred to as the ``bad set" following Epstein \cite{Eps}. In general, one can show that this bad set is closed and nowhere dense (\cite[Section 6]{Eps} for the three dimensional case, \cite[Sections 4 and 6]{EMS} in general). We will assume that $B_1\neq \emptyset$, this is satisfied by any counterexample to the periodic orbit conjecture. Even if the lengths of the orbits depend on a fixed Riemannian metric, the fact that they admit an upper bound does not depend on the choice of metric. In dimension two, the foliation is of codimension one and an upper bound always exists \cite{Reeb,EMS}. Epstein \cite{Eps} proved that in dimension three there are no counterexamples to the periodic orbit conjecture. Summarizing, we can assume that the manifold is of dimension four or higher. \\

Assume further that $X$ is Eulerisable (or admits a strongly-adapted one-form). Then there is some one-form $\alpha$ such that 
\begin{equation}
\begin{cases}
\alpha(X)>0 \\
\iota_Xd\alpha=-dB
\end{cases}
\end{equation}
for some function $B\in C^\infty(M)$. We can suppose that $B$ is not constant (even though we do not need it in the proof), since otherwise $X$ would be geodesible by Gluck's characterization. We know that this cannot be the case due to Wadsley's theorem.

Foliations all whose leaves are compact were studied in depth by Edwards-Millett-Sullivan. We state here a version of the ``moving leaf proposition", confer \cite[Section 5]{EMS}.

\begin{prop}[Moving leaf proposition]\label{prop:moving}
Assume $B_1$ is compact (this is satisfied when $M$ is compact) and non-empty. Then there is a embedded family of leaves with trivial holonomy $L_t \subset M\setminus B_1$, $t\in [0,\infty)$ such that
\begin{itemize}
\item $L_t$ approaches $B_1$ in the sense that given any neighborhood $U$ of $B_1$, there is some $\hat t$ such that $L_{\hat t} \subset U$,
\item $\operatorname{length}(L_t)\rightarrow \infty$ when $t\rightarrow \infty$.
\end{itemize}
\end{prop} 

Let $L_{i}$ be an arbitrary sequence of leaves such $L_i$ approaches $B_1$ in the previous sense. Then Section 3 in \cite{EMS} shows how to find a subsequence (that we still denote $L_{i}$) and a sequence of natural numbers $n_i$ such that 
$ \langle \frac{1}{n_i} L_i, \cdot \rangle$ converges to a positive foliation current $\eta: \mathcal{Z}^1(M)\rightarrow \R$ supported in $B_1$. This current is in fact a cycle, as explained at the end of Section 2 in \cite{EMS}.

We will now choose a suitable initial sequence $L_i$ of leaves contained in family $\{L_t \enspace | \enspace t\in[0,\infty)\}$, i.e. each $L_i$ corresponds to $L_{t_i}$ for some $t_i \in [0,\infty)$. We want the sequence $L_{t_i}$ to satisfy the property that for each $i$
\begin{equation}\label{eq:monot}
\len(L_{t_i})\geq \len(L_{t}) \text{ for all }t\in [0, t_i].
\end{equation} This can be done for the following reason. By Proposition \ref{prop:moving}, $L_t$ is a locally compact invariant subset of $M$ with trivial holonomy. Then by Proposition 4.1 in \cite{EMS}, the length function is continuous along $L_t$. We can now choose each $t_i$ so that the above condition is satisfied: simply consider the intervals $t\in [0,i]$, and pick the value $t_i$ for which the length of $L_{t_i}$ is a maximum in $[0,i]$. Such maximum always exists by the extreme value theorem.

From the discussion above, we can find a subsequence of leaves, which we still denote by $L_{t_i}$ to simplify the notation, and some positive integers $n_i$ so that $\lim \langle \frac{1}{n_i} L_{t_i}, \cdot \rangle$ converges to a positive foliation cycle $\eta$ with support in $B_1$.  Since the length of $L_{t_i}$ tends to infinity, the integers $n_i$ satisfy $n_i \rightarrow \infty$.

Consider the sequence of tangent two chains
$$ T_i= \{ L_t \enspace | \enspace t\in [0,t_i]\}, $$
which are increasingly larger parts of the moving leaf family. The sequence of tangent two chains $A_i=\frac{1}{n_i}T_{i}$ satisfies
$$ \lim_{i\rightarrow \infty} \partial A_i= \lim_{i\rightarrow \infty} \frac{1}{n_i}(L_{t_i}-L_0)= \eta, $$
since the initial leaf $L_0$ has finite length and the $n_i$ go to infinity.

\begin{Remark}
As suggested by Gluck \cite{Gl}, a sequence of two chains such as the $(A_i)_{i\in \mathbb{N}}$ readily contradicts Sullivan's characterization (Theorem \ref{thm:sul}). This shows that $X$ is not geodesible and gives an alternative proof of Wadsley's theorem. In our discussion, we carefully chose the family $A_i$ with the additional property \eqref{eq:monot} which will be used to show that $X$ is not Eulerisable.
\end{Remark}

Each $T_i$ is just an embedded family of closed curves, hence diffeomorphic to a cylinder. Denoting $\lambda=\frac{\alpha}{\alpha(X)}$, the exterior derivative of $\alpha$ decomposes as $d\alpha= -\lambda \wedge dB + \omega$ for some two form $\omega$ such that $\iota_X \omega=0$. Observe that the function $B$ is constant along each curve $L_t$: this follows from the fact that $\iota_Xd\alpha=-dB$, which implies that $\iota_XdB=0$. In particular, we can write $B(t)$ when restricting $B$ to the subset $L_t \subset M$. The integral $A_i(d\alpha)$ can be computed as
\begin{align}
A_i (d\alpha) &=\frac{1}{n_i} \int_{T_i} d\alpha \\
   &= \frac{1}{n_i} \int_{T_i} \omega +\frac{1}{n_i} \int_{T_i} -\frac{1}{\alpha(X)}\alpha \wedge dB\\
   &= \frac{1}{n_i} \int_{T_i} -\frac{1}{\alpha(X)}\alpha \wedge dB, \label{eq:integral}
\end{align}
where we used that $\iota_X\omega=0$.  We want to prove that there is some subsequence $A_{i_r}$ of $A_i$ such that $A_{i_r}(d\alpha)\rightarrow 0$. This would lead to a contradiction because $A_{i_{r}}(d\alpha)=\partial A_{i_r}(\alpha)\rightarrow \eta(\alpha)$, but $\eta$ is a foliation cycle and $\alpha$ is positive on $X$ implying that $\eta(\alpha)>0$.\newline

We will use the motonicity property \eqref{eq:monot} and the lemmas below to prove the existence of the subsequence $A_{i_r}$.

\begin{lemma}\label{lem:aux1}
 For a given arbitrary $\varepsilon>0$, there is some $t_l\in [0,\infty)$ and a sequence $(t_{i_k})_{k\in \mathbb{N}} \subset (t_l,\infty)$ such that $i_k\rightarrow \infty$ and $|B(t_{i_k})-B(t_l)| < \varepsilon$ for all $k$. Furthermore, we can choose the sequence $(t_{i_k})_{k\in \mathbb{N}}$ such that for all $k\in \mathbb{N}$ we have
$$ \frac{\operatorname{length}(L_{t_l})}{n_{i_k}} < \varepsilon. $$

\end{lemma}
\begin{proof}
Fix some $\varepsilon>0$. The function $B$ is a smooth function on the compact manifold $M$ and hence is bounded. In particular there is some constant $K$ such that $|B(t)|<K$ for all $t\in [0,\infty)$. The sequence of numbers $B(t_i)$ is a bounded sequence, so there is a subsequence which is convergent to $y \in \mathbb{R}$. By continuity of $B$, $B(t)\rightarrow y$ and we can find a $t_l\in (t_i)_{i\in \mathbb{N}}$ such that $|B(t_l)-y|<\frac{\varepsilon}{2}$. Similarly, we can find a subsequence $t_{i_k} \rightarrow \infty$ such that $|B(t_{i_k})-y|<\frac{\varepsilon}{2}$ for each $k$. In particular $|B(t_{i_k})-B(t_l)|\leq |B(t_{i_k})-y| +|y-B(t_l)|<\varepsilon$ as we wanted.

Recall that the sequence $n_i$ satisfies that $n_i \rightarrow \infty$. Since $t_l$ is fixed for the given $\varepsilon$, taking $k$ greater than some $k_0$ we can ensure that
$$ \frac{\operatorname{length}(L_{t_l})}{n_{i_k}}<\varepsilon.$$
The value of $\varepsilon$ was arbitrary, so this concludes the proof of the lemma.
\end{proof}

\begin{lemma}\label{lem:aux2}
There exist positive constants $C_1$ and $C_2$ such that
\begin{itemize}
\item $\frac{1}{\min_{M}|X|}|B(t)-B(0)|<C_1$ for any $t\in [0,\infty)$,
\item $\frac{1}{\min_M|X|} \frac{\operatorname{length}(L_{t_i})}{n_i} < C_2$ for every $i$ big enough.
\end{itemize}
\begin{proof}
The vector field $X$ is non-vanishing, so the value of $\frac{1}{\operatorname{min}_M|X|}$ is a well defined positive number. Let $K$ be an upper bound of $|B|$ over $M$. Then 
$$
\frac{1}{\min_{M}|X|}|B(t)-B(0)|<C_1=\frac{1}{\min_{M}|X|}2K,$$ 
for any value of $t$. \newline

For the second inequality, we simply use the fact that $\lim \langle \frac{1}{n_i} L_{t_i},\cdot \rangle$ tends to a smooth  positive foliation cycle.  This implies that $\frac{\operatorname{length}(L_{t_i})}{n_i}$ is a bounded sequence, let $K_2$ be some upper bound. Then $C_2=\frac{1}{\min_{M}|X|}K_2$ satisfies the second inequality.
\end{proof}
\end{lemma}
Given an arbitrary $\varepsilon$, construct a sequence $t_{i_k}$ using Lemma \ref{lem:aux1}. The cylinder $T_{i_k}=\{L_{t}\enspace | \enspace t\in[0,t_{i_k}]\}$, can be parametrized in the following way. Assign smoothly to each leaf $L_s$ with $s\in [0,t_{i_k}]$ a point $p_s\in L_s$. Then each point $q\in L_s$ is given by the flow at some time $\tau_q$ applied to the point $p_s$ . That is, if $\varphi_t$ denotes the flow of $X$, we have
$$ q= \varphi_{\tau_q}(p_s).$$
Taking the function $\tau_q$ smoothly across the leafs, we obtain a function $\tau$ defined in $T_{i_k}$. This shows that each point $q$ in $T_{i_k}$ is given by a pair $(s,\tau)$, where $s\in [0,t_{i_l}]$ and $\tau \in [0, l(s)]$. Here $l(s)$ is the period of $X$ on the leaf $L_s$, hence $\varphi_{l(s)}(p)=p$. We denote by $\theta$ the coordinate $\tau$ because of this periodicity property. The vector field $X$ satisfies by construction $X=\pp{}{\theta}$. We can now explicitly write the integral $A_{i_k}(d\alpha)$ using equation \eqref{eq:integral}.
\begin{align*}
|A_{i_k}(d\alpha)| &=\frac{1}{n_{i_k}} \Big| \int_{0}^{t_l} \int_{L_t} -\frac{1}{\alpha(X)}\alpha \wedge dB + \int_{t_l}^{t_{i_k}} \int_{L_t} -\frac{1}{\alpha(X)}\alpha \wedge dB \Big| \\
			&\leq \frac{1}{n_{i_k}}\Big| \int_{0}^{t_l} \int_{L_t} -\frac{1}{\alpha(X)}\alpha \wedge dB \Big| +  \frac{1}{n_{i_k}}\Big| \int_{t_l}^{t_{i_k}} \int_{L_t} -\frac{1}{\alpha(X)}\alpha \wedge dB \Big| \\
			&= \frac{1}{n_{i_k}}\Big| \int_{0}^{t_l} \int_{L_t} -\frac{\alpha(\tpp{}{\theta})}{\alpha(X)} dB(\tpp{}{s}) d\theta ds\Big| +  \frac{1}{n_{i_k}}\Big| \int_{t_l}^{t_{i_k}} \int_{L_t} -\frac{\alpha(\tpp{}{\theta})}{\alpha(X)} dB(\tpp{}{s}) d\theta ds \Big|\\
			&=  \frac{1}{n_{i_k}}\Big| \int_{0}^{t_l}  -dB(\tpp{}{s}) (\int_{L_s} d\theta) ds\Big| +  \frac{1}{n_{i_k}}\Big| \int_{t_l}^{t_{i_k}} -dB(\tpp{}{s}) (\int_{L_t}  d\theta) ds \Big|   \\
			&\leq  \frac{\operatorname{length}(L_{t_l})}{n_{i_k}\min_M |X|} |B(t_l)-B(0)|+ \frac{\operatorname{length}(L_{t_{i_k}})}{n_{i_k}\min_M |X|}|B(t_k)-B(t_l)| \\
			&< C_1\varepsilon + C_2 \varepsilon.
\end{align*}
We used the triangle inequality, the fact that $dB(\partial_{\theta})=0$, Lemmas \ref{lem:aux1} and \ref{lem:aux2} and equation \eqref{eq:monot} in this computation. The initial $\varepsilon$ was arbitrary while $C_1, C_2$ are fixed and do not depend on $\varepsilon$. 

Take a sequence of positive numbers $\varepsilon_r \rightarrow 0$. For each $\varepsilon_r$, we can use Lemma \ref{lem:aux1} to construct a subsequence $r_k$ satisfying $A_{r_k}(d\alpha)< C_1\varepsilon_r+C_2\varepsilon_r$ for all $k\in \mathbb{N}$. Using a diagonal argument, we choose for each $r$ an element $i_r \in (r_k)_{k\in \mathbb{N}}$ such that $i_r \rightarrow \infty$ as $r\rightarrow \infty$. The resulting subsequence $i_r$ satisfies that $A_{i_r}(d\alpha) < C_1 \varepsilon_r + C_2 \varepsilon_r$, since each $i_r$ is in $(r_k)$. We deduce that $A_{i_r}(d\alpha)\rightarrow 0$ when $r\rightarrow \infty$, leading to a contradiction with the fact that $A_{i_r}(d\alpha)=\partial A_{i_r}(\alpha)\rightarrow \eta(\alpha)>0$ when $r\rightarrow \infty$. This proves Theorem \ref{main}.

\begin{Remark}
Our conclusion, even if it leads to a contradiction, does not explicitly contradict Theorem \ref{thm:eulerisable}. We obtained a sequence of two chains whose flux goes to zero and whose boundary approximates a foliation cycle. An explicit contradiction with Theorem \ref{thm:eulerisable} would require to find a family of exactly zero flux two-chains whose boundary approximates a foliation cycle. This zero-flux sequence can be obtained as in \cite[Section 2, page 5]{PRT} via a modification of the $2$-chains $A_{i_r}$, by substracting some weighted $2$-chain which cancels the flux of $A_{i_r}$.
\end{Remark}

\section{Wadsley's theorem and hydrodynamics}

The classical characterization of vector fields satisfying the periodic orbit conjecture, that we previously cited and which is due to Wadsley \cite{W}, has the following precise statement.
\begin{theorem}
Let $M$ be a compact manifold. A non-vanishing vector field $X$ all whose orbits are closed is geodesible if and only if the set of lengths of the orbits of $X$ admits an upper bound.
\end{theorem}

Taking a careful look at Wadsley's averaging argument, we can show that any vector field satisfying the periodic orbit conjecture is a Beltrami-type stationary solution to the Euler equations for some metric. Not every geodesible vector field is volume-preserving, and notice that we do not impose this condition.
\begin{prop}\label{prop:s1euler}
Let $X$ be a vector field all whose orbits are closed on a compact manifold $M$. If the set of lengths of the orbits admits an upper bound, then $X$ is a Beltrami-type stationary solution to the Euler equations for some metric.
\end{prop}
\begin{proof}
Let us recall the averaging process in \cite[Section 3]{W}. Since the lengths of the orbits are bounded, there exists some positive function $f\in C^\infty(M)$ (see for example \cite[page 218]{Be}) such that the flow of $\tilde X=fX$ defines an $S^1$-action on $M$:
\begin{align*}
\rho: S^1\times M &\longrightarrow M\\
		(\theta,p) &\longmapsto \rho_{\theta}(p).
\end{align*}
Choose any metric $g_1$ in $M$, we can average it over the action with respect to the Haar measure on $S^1$, obtaining another metric
$$ g_2=\int \rho^*g_1.$$
By construction, this metric is invariant by the flow of $\tilde X$ and so $\tilde X$ is a Killing vector field with respect to $g_2$. By Lemma 3.1 in \cite{W}, there is some other metric $g_3$, conformal to $g_2$, such that $\tilde X$ remains Killing with respect to $g_3$ and $g_3(\tilde X,\tilde X)=1$. In addition, the trajectories of $\tilde X$ are geodesics with respect to $g_3$ parametrized by arc-length. This implies that the one-form $\alpha=g_3(\tilde X, \cdot)$ satisfies 
\begin{equation}\label{eq:oneform}
\begin{cases}
\alpha(\tilde X)=1 \\
\iota_{\tilde X}d\alpha=0
\end{cases}.
\end{equation}
Furthermore, since $\tilde X$ is Killing, it preserves the Riemannian volume induced by $g_2$ which we denote by $\mu_2$. Since $\tilde X$ is just a reparametrization of $X$, Equation \eqref{eq:oneform} implies that $\alpha(X)>0$ and $\iota_{X}d\alpha=0$. On the other hand, $X$ preserves the volume form $\mu=f\mu_2$ since 
$$\mathcal{L}_X\mu=d \iota_X\mu=d(\iota_{\frac{1}{f}\tilde X}f\mu_2)=d\iota_{\tilde X}\mu_2=0,$$ where we used that $\tilde X$ preserves $\mu_2$. To conclude, we construct a metric $g$ such that
\begin{itemize}
\item $g(X,\cdot)=\alpha$,
\item $X$ is orthogonal to $\ker \alpha$,
\item the conformal factor of the metric on $\ker \alpha$ is taken such that the Riemannian volume is $\mu$.
\end{itemize}
For this metric, the vector field $X$ is a Beltrami-type stationary solution to the Euler equations.
\end{proof}
In particular, vector fields satisfying the periodic orbit conjecture are characterized by the property of being a Beltrami-type stationary Euler flow, since the converse statement follows immediately from Wadsley's theorem. Explicit examples of $S^1$-actions without fixed points are easy to construct, e.g. the Hopf field on an odd-dimensional sphere (which is already a Beltrami-type stationary Euler flow for the round metric). Corollary \ref{coro} follows from Theorem \ref{main} and Proposition \ref{prop:s1euler}.


\begin{thebibliography}{99}

\bibitem{AK} V.I. Arnold, B. Khesin. \emph{Topological Methods in Hydrodynamics}. Springer, New York, 1999.

\bibitem{Be} A.L. Besse. \emph{Manifolds all of whose geodesics are closed}. Volume 93 of Ergebnisse der Mathematik und ihrer Grenzgebiete [Results in Mathematics and Related Areas]. Springer-Verlag, Berlin-New York, 1978. With appendices by D.B.A. Epstein, J.-P. Bourguignon, L. B\'erard-Bergery, M. Berger and J.L. Kazdan.

\bibitem{C1} R. Cardona. \emph{Steady Euler flows and Beltrami fields in high dimensions}. Ergodic Theory Dynam. Systems (2020), 1-24. doi:10.1017/etds.2020.124.

\bibitem{CV} K. Cieliebak, E. Volkov. \emph{A note on the stationary Euler equations of hydrodynamics}. Ergodic Theory Dynam. Systems, 37 (2017), 454-480. 

\bibitem{EMS} R. Edwards, K. Millett, D. Sullivan. \emph{Foliations with all
leaves compact}. Topology 16(1977), 13-32.

\bibitem{Eps} D.B.A. Epstein. \emph{Periodic flows on $3$-manifolds}. Ann. of Math. 95 (1972), 68-82.

\bibitem{EV} D.B.A. Epstein, E. Vogt. \emph{A Counterexample to the Periodic Orbit Conjecture in Codimension 3}. Ann. of Math., (1978), 108(3), second series, 539-552. 

\bibitem{Gl} H. Gluck. \emph{Dynamical behavior of geodesic fields}. Global theory of dynamical systems (Proc. Internat. Conf.,
Northwestern Univ., Evanston, Ill., 1979), Lecture Notes in Math., 819, 190-215, Springer, 1980.

\bibitem{Morg} F. Morgan. {\it Geometric measure theory: a beginner's guide}. 5th edition. Academic press, 2016.

\bibitem{MS} P. Mounoud, S. Suhr. \emph{Pseudo-Riemannian geodesic foliations by circles}. Math. Z. 274, 225-238 (2013).

\bibitem{PRT} D. Peralta-Salas, A. Rechtman, F. Torres de Lizaur. \emph{A characterization of 3D Euler flows using commuting zero-flux homologies}. Ergodic Theory Dynam. Systems (2020), 1-16. doi:10.1017/etds.2020.25.

\bibitem{R} A. Rechtman. \emph{Use and disuse of plugs in foliations}. PhD Thesis, ENS Lyon, 2009.

\bibitem{Reeb} G. Reeb. \emph{Sur certaines propiétés topologiques des variétés feuilletées}. Actual. scient. ind. 1183 (1952).

\bibitem{S1} D. Sullivan. \emph{A counterexample to the periodic orbit conjecture}.
Publ. IHES 46(1976), 5-14.

\bibitem{S2} D. Sullivan. \emph{A foliation of geodesics is characterized by having no ``tangent homologies"}. J. Pure
Appl. Algebra 13 (1978), no. 1, 101-104.

\bibitem{Tao} T. Tao. \emph{On the universality of potential well dynamics}. Dyn. PDE 14 (2017), 219-238.

\bibitem{W} A.W. Wadsley. \emph{Geodesic foliations by circles}. J. Diff. Geom.
10(1975), 541-549.

\end{thebibliography}
\end{document}